\documentclass[11pt,reqno]{amsart}
\usepackage{fullpage,a4wide}
\usepackage{youngtab}
\usepackage{graphicx}
\usepackage{hyperref}
\usepackage{amsmath,amsopn,amssymb,amsfonts,stmaryrd}
\usepackage{verbatim}
\usepackage{amsthm}
\usepackage{mathtools}
\usepackage{color}
\usepackage{enumitem}
\usepackage[framemethod=TikZ]{mdframed}
\usepackage{bbm}
\usepackage{mathrsfs}
\usepackage{booktabs}
\usepackage{bm}
\usepackage{tensor}
\usepackage{cleveref}
\usepackage{float}

\theoremstyle{plain}
\newtheorem{theorem}{Theorem}[section]
\newtheorem{corollary}[theorem]{Corollary}
\newtheorem{lemma}[theorem]{Lemma}
\newtheorem{proposition}[theorem]{Proposition}
\newtheorem{conjecture}{Conjecture}

\theoremstyle{remark}
\newtheorem*{remark*}{Remark}
\newtheorem*{remarks*}{Remarks}

\theoremstyle{definition}

\newtheorem*{notation*}{Notation}
\newtheorem*{claim*}{Claim}
\newtheorem*{Conjecture1*}{Conjecture 1}
\newtheorem*{Conjecture2*}{Conjecture 2}
\newtheorem*{Conjecture3*}{Conjecture 3} 
\newtheorem*{Conjecture4*}{Conjecture 4}

\numberwithin{equation}{section}

\setlist[itemize]{noitemsep, topsep=0pt}

\makeatletter
\newcommand{\vast}{\bBigg@{3}}
\newcommand{\Vast}{\bBigg@{5}}
\makeatother

\allowdisplaybreaks

\def\func#1{\mathop{\rm #1}}%

\makeatletter
\@namedef{subjclassname@2020}{%
	\textup{2020} Mathematics Subject Classification}
\makeatother
\begin{document}
\title{Polynomization of Sun's Conjecture}
\author{Bernhard Heim}
\address{Department of Mathematics and Computer Science\\Division of Mathematics\\University of Cologne\\ Weyertal 86--90 \\ 50931 Cologne \\Germany}
\address{Lehrstuhl A f\"{u}r Mathematik, RWTH Aachen University, 52056 Aachen, Germany}
\email{bheim@uni-koeln.de}
\email{bernhard.heim@rwth-aachen.de}
\author{Markus Neuhauser}
\address{Kutaisi International University, 5/7, Youth Avenue,  Kutaisi, 4600 Georgia}
\address{Lehrstuhl
f\"{u}r Geometrie und
Analysis, RWTH Aachen University, 52056 Aachen, Germany}
\email{markus.neuhauser@kiu.edu.ge}
\begin{abstract}
Let $p(n)$ denote
the number of partitions of a natural number $n$. 
As $ n \to \infty$, the $n$th root
of $p(n)$ tends to $1$,
which is related to the Cauchy--Hadamard test for power series. Andrews also discovered an elementary proof.
Sun conjectured that this happens in a certain way for $n\geq 6$:
\begin{equation*}
\sqrt[n]{p(n)} > \sqrt[n+1]{p(n+1)}.
\end{equation*}
The conjecture was proved by Wang and Zhu; shortly thereafter, Chen and Zheng independently obtained a second proof.
 In this paper, we follow an approach by
 Rota. We consider $p(n)$ as special values of the D'Arcais polynomials, known as the Nekrasov--Okounkov polynomials.  This identifies Sun's conjecture as a property of the largest real zero of certain polynomials. This leads to results towards $k$-coloured partitions, overpartitions, and plane partitions. Moreover, we also consider Chebyshev and Laguerre polynomials. The main purpose of this paper is to offer a
 uniform approach.
\end{abstract}
\maketitle
\section{Introduction and statement of results}
In this paper, we generalize Sun's conjecture \cite{Su13} on the decreasing monotonicity of the $n$th root of the partition function $p(n)$. Sun conjectured (\cite{Su13}, Conjecture 2.14 (2012-08-02)) that for all $n\geq 6$ we have
\begin{equation}\label{SunC}
\sqrt[n]{p(n)} > \sqrt[n+1]{p(n+1)}.
\end{equation}

The conjecture has been proven by Wang and Zhu \cite{WZ14}
and shortly later by Chen and Zheng \cite{CZ17}. Both proofs 
depend on the celebrated circle method of Hardy--Ramanujan and the Hardy--Ramanujan--Rademacher formula for $p(n)$ and are closely related to
the work by
Nicolas \cite{Ni78} and DeSalvo and Pak \cite{DP15}
on the log-concavity of $p(n)$.
In this paper, we take a new approach. 

In the spirit of Rota \cite{RS07}, we view Sun's conjecture as a claim on the location of
zeros of certain associated
polynomials. 
\begin{center}{\it
``The one contribution of mine that I hope will be remembered
has consisted in just pointing out that all sorts of problems of
combinatorics can be viewed as problems of location of the zeros of
certain polynomials and in giving these zeros a combinatorial interpretation.
This is now called the critical problem.''}
\end{center}
The approach provided in this paper is the following.
Let $g$ be a normalized arithmetic function with positive values.
We assign a sequence of polynomials $\{P_n^g(x)\}_{n\geq 0}$:
\begin{equation}\label{polynomials}
P_n^g(x):= \frac{x}{n} \sum_{k=1}^n g(k) P_{n-k}^g(x)
\end{equation}
with initial value $P_0^g(x)=1$. Then, the location of the largest real zero $x_n^g$ of
\begin{equation}\label{delta}
\Delta_n^g(x):= \left(P_n^g(x)\right)^{n+1}-\left(P_{n+1}^g(x)\right)^n,
\end{equation}
encodes the inequality (\ref{SunC}). 
\subsection{\label{abschn1.1}D'Arcais
polynomials}
The polynomials associated with
$g(n) = \sigma(n)=\sum_{d \mid
n}$ are the D'Arcais polynomials \cite{DA13,HN20}. In combinatorics and statistical mechanics, they are
called Nekrasov--Okounkov polynomials \cite{NO06,Ha10} and provide a
hook length formula. We have
\begin{equation*}\label{product}
\sum_{n=0}^{\infty} P_n^{\sigma} (x) \, t^n =
\prod_{n=1}^{\infty} \Big( 1 - t^n \Big)^{-x} =
\sum _{n=0}^{\infty }
\sum_{\lambda \vdash 
n} 
t^{\vert \lambda \vert} \,
\prod_{ h \in
{H}(\lambda)} \left(   1+ \frac{x-1}{h^2}\right).
\end{equation*}
Here, $\lambda 
\vdash n$ denotes that 
$\lambda $ is a partition 
of $n$ and 
${H}(\lambda)$ be the multiset of hook lengths associated
with $\lambda \vdash 
n$ and $|\lambda|$ is the weight of $\lambda$.
Therefore, $p(n)=P_n^{\sigma}(1)$, and more generally $P_n^{\sigma}(k)$
is the
number of $k$-colored partitions of $n$.
By Lemma \ref{lemma:largest}
we have
\begin{equation*}\label{infinity}
\lim_{x \to \infty} \Delta_n^g(x) = \infty.
\end{equation*}
In Figure \ref{figure1}, we have plotted
the largest real zeros of $\Delta_n^{\sigma}(x)$ for $1 \leq n \leq 25$.
\begin{figure}[H]
\includegraphics[width=.5\textwidth]{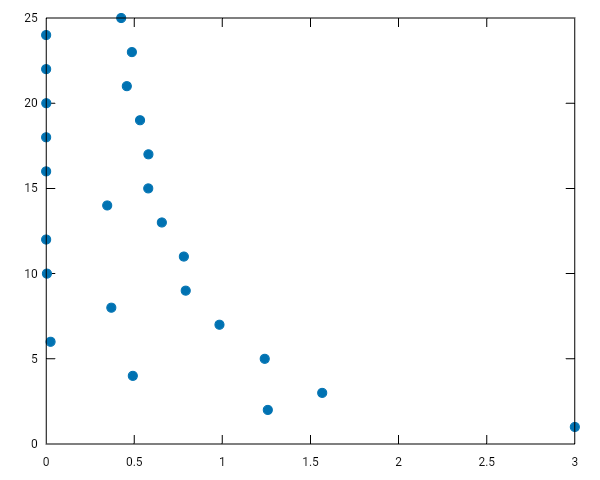}
\caption{\label{figure1}Largest real zeros of $\Delta_n^{\sigma}(x)$
for $1 \leq n \leq 25$}
\end{figure}
We extended
our calculations with the support of Steven Charlton and obtain that $x_n^{\sigma}<1$ for $6\leq n\leq 150$ and speculate that this
holds true for all $n\geq 6$. 
We refer to Section \ref{Section:Arcais} for a full picture of our results related to
partition numbers and the D'Arcais polynomials.
\subsection{Pochhammer und Laguerre
polynomials} \label{sect:pola}
Already simple input functions $g(n)$ lead to interesting results and give
evidence on the
significance of studying the largest real zeros $x_n^g$.
Let $\psi_{\ell}(n):= n^{\ell}$ for $\ell \in \mathbb{N}_0$.
We start with the case $\ell=0$, which is related to
the Pochhammer polynomials.
\begin{theorem} \label{Pochhammer}
Let $g(n)= \psi_0(n)$. Then we obtain
\begin{equation*}
P_n^{\psi_0}(x) = \frac{1}{n!} \prod_{k=0}^{n-1} \left( x+k \right)
\text{ and } x_n^{\psi_0}=\psi_0(2)=1 \text{ for } n\geq 2.
\end{equation*}
\end{theorem}
The proof is given in Section \ref{ell=0}. This result implies
that $P_n^{\psi_0}(x)^{n+1}> P_{n+1}^{\psi_0}(x)^{n}$
for all $x>1$ and $n \geq 2$. 
More involved is the next case, $g(n)=\psi_1(n)$, which is related to orthogonal polynomials, the $\alpha$-associated Laguerre polynomials for $\alpha=1$:
\begin{equation*}\label{alpha=1}
L_n^{(\alpha)}(x) :=  \sum_{k=0}^n \binom{n+\alpha}{n-k}  \frac{(-x)^k}{k!}.
\end{equation*}
From \cite{HLN19}, we have 
$P_n^{\psi_1}(x) = \frac{x}{n} L_{n-1}^{(1)} (-x)$.
Utilizing a
Tur\'{a}n inequality of $P_n^{\psi_1}(x)$ \cite{HNT22} and an
initial property of $\Delta_n^{\psi_1}(x)$ we will prove in Section \ref{ell=1}
the following result.
\begin{theorem} \label{th:ell=1}
Let $g(n)= \psi_1(n)$. Then $P_n^{\psi_1}(x) = \frac{x}{n} L_{n-1}^{(1)} (-x)$ for $n\geq 1$. The largest real
zero of $\Delta_n^{\psi_1}(x)$ satisfies $x_n^{\psi_1} \leq
{\psi_1}(2)=2$ for all $n\in \mathbb{N}$. Further, for all $n \geq 6$, we have $x_n^{\psi_1} \leq 1$.
\end{theorem}
This implies that for all $x \geq 1$ and $n\geq 6$:
\begin{equation*}
\frac{x^{n+1}}{n^{n+1}} \, L_{n-1}^{(1)}(-x)^{n+1} >  
\frac{x^{n}}{(n+1)^{n}} \, L_{n}^{(1)}(-x)^{n}.
\end{equation*}
It would be interesting to invest in this new type of inequality for
other orthogonal polynomials, such
as Chebyshev and Hermite polynomials.
\subsection{Plane
partitions and
overpartitions}
\subsubsection{Plane
partitions}
Recently, Ono, Pujahari, and Rolen \cite{OPR22} proved a conjecture
\cite{HNT23} on the log-concavity of the sequence of plane 
partition numbers $\{pp(n)\}_n$ for $n\geq 12$. This strongly suggests to 
get similiar results for associated
polynomials $P_n^{\sigma_2}(x)$ as we get for D'Arcais polynomials $P_n^{\sigma}(x)$. Nevertheless, we have, at the moment, only partial results, although our numerical data draws a clear picture.
Plane partitions are fascinating generalizations of 
integer partitions
(\cite{St99}, Section 7.20). Andrews \cite{An98}
also introduced
plane partitions in the context of higher-dimensional partitions.

We follow the introduction on plane partitions provided
in \cite{HNT23}.
A plane partition $\pi$ of $n$ is an array 
$\pi = \left( \pi_{ij} \right)_{i,j \geq 1}$ of non-negative integers $\pi_{ij}$,
with finite sum $\vert \pi \vert := \sum_{i,j\geq
1} \pi_{ij}=n$, which is weakly decreasing in rows and columns.
It can be considered as the filling of a Ferrers diagram with weakly decreasing rows and columns,
where the sum of all these numbers
equals
$n$.
Let the numbers in the filling represent the 
heights for stacks of blocks placed
at the corresponding cell
of the diagram (Figure \ref{cube}). We refer to \cite{HNT23}.
This is a natural generalization of the concept of
classical partitions \cite{An98, On04}.
%
\begin{figure}[H]
\begin{minipage}{0.35\textwidth}
\begin{equation*}
\phantom{xxx}
\young(5443321,432,21) 
\end{equation*}
\end{minipage}
$\longrightarrow$ \phantom{xx}
\begin{minipage}{0.4\textwidth}\phantom{xx}
\includegraphics[width=0.75\textwidth]{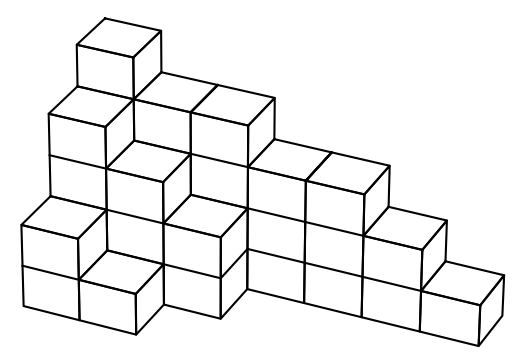}  
\end{minipage}
\caption{\label{cube}Representation of plane partitions}
\end{figure}
Let $\mathop{pp}\left( n\right) $
denote
the number of plane partitions of $n$. Similarly
as Euler discovered the
generating function for the partition numbers $p(n)$, MacMahon discovered
the generating function for the plane partitions $\mathop{pp}\left( n\right)$:
\begin{equation*}
\prod_{n=1}^{\infty} \left( 1 - q^n\right)^{-n} = \sum_{n=0
}^{\infty} 
\mathop{pp}\left( n\right) \, q^n.
\end{equation*}
Therefore, let $g(n)= \sigma_2(n)$, where $\sigma_{\ell}(n):= \sum_{d \mid
n} d^{\ell}$. Then we have
$\mathop{pp}\left( n\right) = P_n^{\sigma_2}(1)$. 
We examine
the first values of $\sqrt[n]{\mathop{pp}\left( n\right) }$, see
Table \ref{R}.
\begin{table}[H]
\[
\begin{array}{c|cccccccccc}
\hline
n&1&2&3&4&5&6&7&8&9&10\\ \hline \hline
\sqrt[n]{\mathop{pp}
\left( n\right) }
&1 &1.732    &1.817
&1.899
&1.888&
1.906&1.890
&1.886
&1.872
& 1.862
\\ \hline
\end{array}
\]
\caption{\label{R}Approximative values of $\sqrt[n]{\mathop{pp}\left( n\right) }$}
\end{table}
Further, we have plotted
the largest real
zeros of $\Delta_n^{\sigma_2}(x)$ for $1 \leq n \leq 25$ (see Figure \ref{figure:plane}).
Then it is obvious to expect:
\begin{conjecture}
Let $g(n)$ be equal to $\sigma_{2}(n)$. Then
we have that 
$\Delta_n^{\sigma_2}(x)>0$ for $x \geq g
\left( 2
\right) =5$ and all $n \in \mathbb{N}$. 
Moreover, for $n \geq 6$, we already have $\Delta_n^{\sigma_2}(x)>0$ for $x \geq 1$.
\end{conjecture}
\begin{figure}[H]
\includegraphics[width=.5\textwidth]{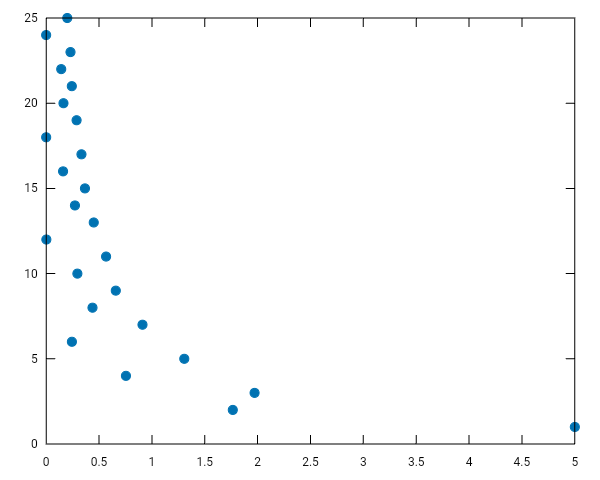}
\caption{\label{figure:plane}Largest real zeros of $\left( P_{n}^{\sigma_2}\left( x\right) \right) ^{n+1}-\left( P_{n+1}^{\sigma_2}\left( x\right) \right) ^{n}$ with $n$ vertical}
\end{figure}
We prove the conjecture for $x=1$ utilizing the log-concavity of 
$\mathop{pp}\left( n\right)$ for $n\geq 12$ by a similar argument as
given in \cite{WZ14}.
\begin{theorem} \label{th:plane}
The sequence
$\left\{ \sqrt[n]{\mathop{pp}\left( n\right) }\right\} _{n}$
is strictly decreasing for $n\geq 6$.
\end{theorem}
\subsubsection{Overpartitions}
In our investigation of overpartitions, we have observed analogous patterns.
An overpartition of $n$ is a non-increasing sequence of natural numbers
that sum to $n$ where the
last occurrence of a number
may be overlined \cite{CL04}.
We denote the number of overpartitions of $n$
by $\bar{p}(n)$. 
The following generating function is well-known:
\begin{equation*}
\sum_{n=0}^{\infty} \bar{p}(n) \, t^n = \prod_{n=1}^{\infty}
\frac{ 1 + t^n}{1-t^n} = 1 + 2t+4t^2 + 8 t^3 + 14 t^4+ \cdots .
\end{equation*}
We offer the following polynomization
\begin{equation*} \label{identity}
\sum_{n=0}^{\infty}   P_n^{\bar{g}}(x) \,t^n = \left(   
\frac{\prod_{n=1}^{\infty} ( 1+ t^{n})}{\prod_{n=1}^{\infty} (1-t^n)}\right)^{\frac{x}{2}} 
= \left( \sum_{n=-\infty}^{\infty} (-1)^n \, t^{n^2} \right)^{- \frac{x}{2}}.
\end{equation*}
In Table \ref{Pn}, we
provide several values of $P_n^{\bar{g}}(x)$.
Note that $\bar{p}(n) = P_{n}^{\bar{g}}(2)$ denotes the number of overpartitions of $n$ and more generally
$P_{n}^{\bar{g}}(2k)$
corresponds to the number of
$k$-colored
overpartitions of $n$ for $k \in \mathbb{N}$.
\begin{table}[H]
\[
\begin{array}{c|ccccccc}
\hline
n & 1 & 2 & 3 & 4 & 5 & 6 & 7 \\ \hline \hline
P_{n}^{\bar{g}}\left( -2\right) & -2& 0& 0& 2& 0& 0& 0 \\
P_{n}^{\bar{g}}\left( -1\right) &-1
&-\frac{1}{2}
&-\frac{1}{2}
&\frac{3}{8}
&\frac{1}{8}
&\frac{3}{16}
&\frac{7}{16} \\
P_{n}^{\bar{g}}\left( 1\right) &1
&\frac{3}{2}
&\frac{5}{2}
&\frac{27}{8}
&\frac{39}{8}
&\frac{111}{16}
&\frac{149}{16} \\
P_{n}^{\bar{g}}\left( 2\right) & 2& 4& 8& 14& 24& 40& 64 \\
\hline
\end{array}
\]
\caption{\label{Pn} Values $P_n^{\bar{g}}(x)$}
\end{table}
The
distribution of the largest zeros $x_n^{\bar{g}}$
(see Figure \ref{overpart})
suggests that $x_n^{\bar{g}}<1$ for $n \geq 6$.
\begin{figure}[H]
\includegraphics[width=.5\textwidth]{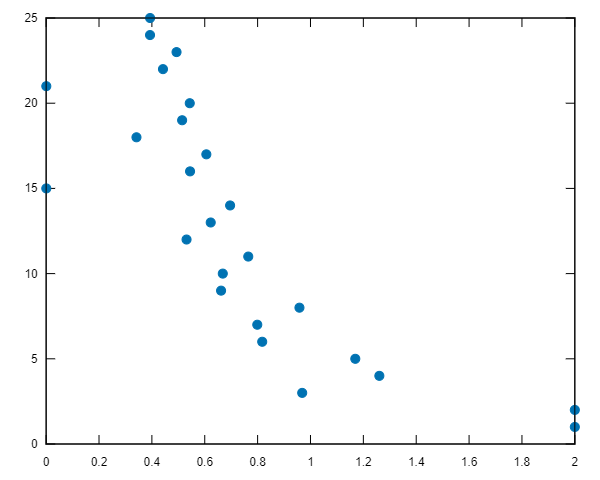}
\caption{\label{overpart}Largest real zeros of $\Delta_n^{\bar{g}}(x)$ for $1 \leq n \leq 25$}
\end{figure}
For $x=2$, we prove:
\begin{theorem} \label{th:over}
We have that $\Delta_1^{\bar{g}}(2) 
= \Delta_2^{\bar{g}}(2)=
0$. Let $n \geq 3$.
Then
\begin{equation*}
\bar{p}(n)^{n+1} > \left(\bar{p}
\left( {n+1}\right) \right)^n.
\end{equation*}
Therefore $\Delta_n^{\bar{g}}(2)>0$ for $n \geq 3$. 
\end{theorem}
\section{Preliminaries}
\subsection{Basic properties of sequences}
Let $\{\alpha(n)\}_{n=0}^{\infty}$ be a sequence of positive real numbers.
We consider roots and quotients of sequence members. This has already been
a successful approach in analysis finding the radius of convergence in the associated power series $\sum_{n=0}^{\infty} \alpha(n)\, t^n$:
\begin{equation*}
R_{\alpha}(n):= \sqrt[n]{\alpha(n)}\,\, \text{ and }\,\, Q_{\alpha}(n):= \frac{\alpha(n)}{\alpha(n-1)}.
\end{equation*}
Then we obtain for $n\geq 1$:
\begin{equation}\label{equalities}
R_{\alpha }\left( n
\right) >Q_{\alpha }\left( n
\right) \Leftrightarrow R_{\alpha }\left( n
-1\right) >R_{\alpha }\left( n
\right) \Leftrightarrow R_{\alpha }\left( n-1\right) >Q_{\alpha }\left( n
\right) .
\end{equation}
A sequence is
called log-concave for $n \geq n_0$ if
\begin{equation*}
\alpha(n)^2 - \alpha(n-1) \alpha(n+1) \geq 0
\end{equation*}
for all $n \geq n_0$. It is
called strictly log-concave if the inequality is strict.
\begin{proposition}\label{prop: strategy}
Let the sequence $\{\alpha(n)\}_n$ of positive real numbers be log-concave 
for all $n \geq n_0$ for $n_0 \in \mathbb{N}$. Then, the initial condition
\begin{equation}\label{special}
R_{\alpha }\left( n_{0} -1\right)
> R_{\alpha }\left( n_{0}\right)
\end{equation}
implies that the sequence $\{\sqrt[n]{\alpha(n)}\}_{n= n_0-1}^{\infty}$ is strictly decreasing.
\end{proposition}
\begin{proof}
A sequence $\{ \alpha(n)\}_n$ is log-concave at $n$ if
and only if
$Q_{\alpha}(n) \geq Q_{\alpha}(n+1)$. From the general property (\ref{equalities}), we observe that the initial condition is equal to $R_{\alpha}(n_0) > Q_{\alpha}(n_0)$. And applying again (\ref{equalities}) leads to $R_{\alpha}(n_0) > R_{\alpha}(n_0+1)$. The proof can be finished
applying mathematical induction. 
\end{proof}
If we consider $ 1 / \alpha(n)$, we obtain similar results for log-convex sequences. Wang and Zhu first observed that the monotonicity
of
quotients implies the monotonicity of the roots stated in (\cite{WZ14}, Theorem 
2.1). Note that Proposition \ref{prop: strategy}
implies the result
by Wang and Zhu.
\subsection{Basic
properties of
generalized D'Arcais
polynomials}
Let $g$ be a normalized arithmetic function with positive values.
We have introduced in the
introduction (see Definition (\ref{polynomials})),
associated with
$g$, the generalized D'Arcais polynomials $\{P_n^g(x)\}_{n\geq 0}$.
Then, the location of the largest real zero $x_n^g$ of
\begin{equation*}
\Delta_n^g(x)= \left(P_n^g(x)\right)^{n+1}-\left(P_{n+1}^g(x)\right)^n,
\end{equation*}
determines the positivity of $\Delta_n^g(x)$ since:
\begin{lemma}\label{lemma:largest}
Let the sequence $\{P_n^g(x)\}_n$ be given. Let $x_n^g$ be the largest real zero of $\Delta_n^g(x)$. Then $\Delta
_{n}^{g}\left( x
\right) >0$ for all $x > x_n^g$.
\end{lemma}
\begin{proof}
Let $n \geq 1$.
Note that real
zeros exist
since $P_n^g(0)=0$. The degree of $P_n^g(x)$ is equal to $n$ and is
determined by $P_n^g(x) = \frac{1}{n!}\sum_{k=1}^n A_{n,k}^g \, x^k$ with $A_{n,n}^g=1$.
Consequently, the leading coefficient of $\Delta_n^g(x)$ is given by:
\begin{equation*}
\left( \frac{1}{n!} \, A_{n,n}^{g} \right) ^{n+1}-\left( \frac{1}{\left( n+1\right) !}A_{n+1,n+1}^{g}\right) ^{n}=
\left(\frac{1}{n!}\right)^{n+1} - \left(\frac{1}{(n+1)!}\right)^n>0.
\end{equation*}
\end{proof}
In this paper, we provide evidence that
$x_n^g \leq g(2)$ for polynomials related to partition numbers, plane partition numbers, overpartitions including
Pochhammer polynomials and associated Laguerre polynomials. We provide proofs in some cases but are far away from the general case. We believe that new techniques have to be discovered.
Nevertheless, for $n=1$ and $n=2$, we are explicit and obtain some necessary conditions.
\begin{remark*}
We have that 
\begin{eqnarray*}
\Delta_1^g(x) &=& \frac{x}{2} \left( x
- g(2)\right) ,\\
\Delta _{2}^{g}\left( x\right) & = &\frac{x^{2}}{72}\left( 7\*x^4
 + 15\*g\left( 2\right) \*x^3
 + \left(9\*\left( g
 \left( 2
 \right) \right) ^2
 - 8\*g\left( 3\right) \right) \*x^2 \right.\\
&&\left. + \left(9\*\left( g
 \left( 2
 \right) \right) ^3
 - 24\*g
 \left( 3\right)
 \*g
 \left( 2\right)
 \right) \*x
 - 8\*\left( g
 \left( 3\right)
 \right) ^2\right). 
 \end{eqnarray*}
\end{remark*}

\begin{lemma}
We have that
$\Delta _{2}^{g}\left( x\right) \geq 0$
for all $x\geq g\left( 2\right) $ if and only if
$g\left( 3\right) \leq
\left( g\left( 2\right) \right) ^{2}$.
\end{lemma}

\begin{proof}
Let $\Delta _{2}^{g}\left( x\right) =\frac{x^{2}}{72}Q\left( x\right) $.
Then
$Q\left( x+g\left( 2\right) \right) =7\*x^4
 + 43\*g\left( 2\right) \*x^3
 + \left(96\*\left( g\left( 2\right) \right) ^2
 - 8\*g\left( 3\right) \right) \*x^2
 + \left(100\*\left( g\left( 2\right) \right) ^3
 - 40\*g\left( 3\right) \*g\left( 2\right) \right) \*x
 + 40\*\left( g\left( 2\right) \right) ^4
 - 32\*g\left( 3\right) \*\left( g\left( 2\right) \right) ^2
 - 8\*\left( g\left( 3\right) \right) ^2
 $.
Suppose
$g\left( 3\right)
\leq \left( g\left( 2\right) \right) ^{2}$.
Then
$96\left( g\left( 2\right) \right) ^{2}-8g\left( 3\right) \geq 88
\left( g\left( 2\right) \right) ^{2}$
and
$100
\left( g\left( 2\right) \right) ^{3}-
40g\left( 3\right) g\left( 2\right) \geq 60\left( g\left( 2\right) \right) ^{3}$.
As
\begin{equation}
40\left( g\left( 2\right) \right) ^{4}-32g\left( 3\right) \left( g\left( 2\right) \right) ^{2}-8\left( g\left( 3\right) \right) ^{2}
=8\left( \left( g\left( 2\right) \right) ^{2}
-g\left( 3\right) \right) 
\left( 5\left( g\left( 2\right) \right) ^{2}
+g\left( 3\right) \right) \label{eq:g2g3}
\end{equation}
we have
$$Q
\left( x\right) \geq
7x^{4}+43
g\left( 2\right) x^{3}+88
\left( g\left( 2\right) \right) ^{2}x^{2}+60
\left( g\left( 2\right) \right) ^{
3}x \text{ for }x\geq 0. $$
If $g\left( 3\right) >\left( g\left( 2\right) \right) ^{2}$ then
$Q\left( 0\right) <0$ by (\ref{eq:g2g3}).
\end{proof}


The arithmetic functions 
$\sigma_{\ell}(n):=\sum_{d \mid
n}n^{\ell}$
and especially the case $\ell=1$ leads to Sun's conjecture. 
In this case, we recover the D'Arcais polynomials \cite{DA13,HN20}, which are related to the $q$-expansion of powers of the Dedekind function. 
These polynomials are also known as the Nekrasov--Okounkov polynomials \cite{NO06,Ha10} in statistical mechanics and combinatorics.
Note that $P_n^{\sigma _{1}}\left( 1\right)
=p(n)$ and $P_n^{\sigma _{2}}\left( 1
\right) ={pp}(n)$, where ${pp}(n)$ is the number of plane partitions of $n$. We refer to
\cite{An98}. Several other interesting cases indicate
a general pattern of $\Delta_n^g(x)$ (\ref{delta}) regarding the location of the largest real zero of
polynomials defined by (\ref{polynomials}).
The number of $\ell$-tupels of permutations in $S_n$, which generate abelian subgroups, is related to 
$g_{\ell}(n):= \sum_{d \mid
n} dg_{\ell-1}(d)$. Here
$g_0(n)=1$ for $n=1$ and $0$ otherwise \cite{HH90,BF98,Ab24,ABDV24}. 
Moreover, let 
$\bar{g}(n):= \sigma_1(n)- \sigma_1(n/2)$ 
with $\sigma(x)=0$ for $x \not\in \mathbb{N}$. Then $P_n^{\bar{g}}(1)=\bar{p}(n)$ \cite{CL04} is the number of overpartitions of $n$.

\section{Partitions and D'Arcais
polynomials} \label{Section:Arcais}
Let $p(n)$ denote the number of partitions of $n$. A partition $\lambda$ of a positive integer $n$ is any non-increasing sequence $(\lambda_1, \lambda_2, \ldots, \lambda_d)$ of positive integers whose sum $\vert \lambda\vert:=\sum_{k=1}^{d}\lambda _{k}$ is equal to $n$. We refer to
the works by Andrews
and Ono
\cite{An98, On04}. Euler's generating function
\begin{equation*}
\sum_{n=0}^{\infty} p(n) \, q^n = \prod_{n=1}^{\infty} \frac{1}{1 -q^n} = 1+q+ 2q^2+ 3q^3 + 5q^4+ \cdots
\end{equation*}
represents a holomorphic function in the circle of
unity (\cite{HW60}, \cite{Ap76}(Section 14.3), \cite{Re98} (Chapter $1$)).
By Cauchy--Hadamard's theorem,
the sequence of 
partition numbers $\{p(n)\}$ (\cite{AE04}, Section 6.4)
exhibits a
subexponential function,
satisfying
\begin{equation}\label{lim}
\lim_{n \to \infty} \sqrt[n]{p(n)} = 1.
\end{equation}
Andrews demonstrated
that this property can also
be derived through
combinatorial methods
\cite{An71,AE04}.
The first values of $p(n)$ (Table \ref{first values}) show some irregularity. 
\begin{table}[H]
\[
\begin{array}{c|ccccccc}
\hline
n&1&2&3&4&5&6&7\\ \hline \hline
\sqrt[n]{p\left( n\right)}
&1.00&1.41&1.44&1.50&1.48&1.49 &1.47 \\ \hline
\end{array}
\]
\caption{\label{first values}
Approximative values of $\sqrt[n]{p\left( n\right)}$}
\end{table}
Sun \cite{Su13} observed and conjectured that the 
sequence $\sqrt[n]{p(n)}$ satisfies the following property
which was proven by Wang and Zhu \cite{WZ14}; shortly thereafter,
Chen and Zheng \cite{CZ17} obtained an independent proof.
\begin{theorem}[Wang and Zhu \cite{WZ14}]\ \\
\label{Andrew}The sequence $\{\sqrt[n]{p(n)}\}_n$ is strictly decreasing for $n \geq 6$. 
\end{theorem}
\begin{proof} We consider the sequence $\sqrt[n]{p(n)}_{n\geq 0}$.
Nicolas \cite{Ni78} (see also DeSalvo and Pak \cite{DP15}) proved the log-concavity of $\{p(n)\}$ for $n \geq 26$. Therefore, we can apply
Proposition \ref{prop: strategy} for $n_0=26$. The remaining cases
$ 6 \leq n \leq 24$ can be checked directly by a computer.
\end{proof}
This proof method we
call the {\it log-concavity method with initial condition\/}, discovered by Wang and Zhu. 
\begin{remark*} 
The only known proof
of the log-concavity of the sequence
$\{p(n)\}_n$ for $n \geq 26$ relies
on the celebrated circle method
by Hardy and Ramanujan \cite{HR17}, utilizing Rademacher's explicit formula \cite{Ra43}. 
\end{remark*}

The strictly decreasing property also reads as
\begin{equation} \label{p:new}
p(n)^{n+1} > p(n+1)^n,
\end{equation}
which may have some combinatorial interpretation.
Encouraged by Andrews' combinatorial proof (\cite{AE04}, Section 6.4) of (\ref{lim}), we believe it is worth examining
whether inequality (\ref{p:new}) also admits a combinatorial proof.
Notably,
combinatorial proofs had been provided
for
the Bessenrodt--Ono inequality \cite{BO16} by Alanazi, Gagola,
and Munagi \cite{AGM17} as well as
by Chern, Fu, and Tang \cite{CFT18}. 
%
%
If one considers the partition numbers $p(n)$ as a special value of
the $n$th D'Arcais polynomial $P_n^{\sigma}(x)$, where $p(n)=P_n^{\sigma}\left( 1\right) $,
it
is natural to invest in the distribution of the zeros  of $\Delta_n^{\sigma}\left( x\right)
$. Since $\lim_{x \to \infty} \Delta_n^{\sigma}(x)= \infty$,
it is obvious that $\Delta_n^{\sigma}(x) >0$ for $x >
x_n^{\sigma}$, where $x_n^{\sigma}$ is the largest real zero of $P_n^{\sigma}
\left( x\right)
$. We refer to Section \ref{abschn1.1}
for a plot of the largest real zero for $1 \leq n \leq 25$.
More generally, we have with the support of Steven Charlton:
\begin{theorem} \label{color} Let $6 \leq n \leq 150$. Then $x_n^{\sigma} <1$.
\end{theorem}
Further calculations lead to the following observation.
\begin{conjecture} 
\label{conj:sigma} Let $n \geq 6$ and $x \geq 1$. Then
\begin{equation*}
(P_n^{\sigma}(x))^{n+1} - (P_{n+1}^{\sigma}(x))^n >0.
\end{equation*}
\end{conjecture}

Let $k \in \mathbb{N}$. We denote by
$p_k(n)$ the number of $k$-colored partitions of $n$ \cite{CFT18, BKRT21, BKPR25}. A $k$-colored partition of
$n$ is a partition of $n$ where to
each part one of $k$ colors is allocated.
Note that $p_1(n)=p(n)$. The D'Arcais polynomials evaluated at $k$ generalize
Euler's original approach towards the number of partitions $p(n)$.
We have $P_n^{\sigma}(k)=p_k(n)$. Therefore, Theorem \ref{color} implies that
$\{\sqrt[n]{p_k(n)}\}_{n}$ is strictly decreasing for $6 \leq n \leq 150$ for all $k$-colored partitions. We prove
Conjecture \ref{conj:sigma} for $k$-colored partitions. We utilize a deep
result by
Bringmann, Kane, Rolen,
and Tripp \cite{BKRT21} confirming the
Chern--Fu--Tang conjecture \cite{CFT18}. We refer also to (\cite{HN21},
(1.2)).
\begin{theorem} 
\label{th:color}Let $n,k \in \mathbb{N}$ and $n \geq 5
$.
Then $\Delta_n^{\sigma}(k) >0$. Therefore, for all
$k$-colored partitions and all $n \geq 5
$:
\begin{equation} \label{discrete:result}
p_k(n)^{n+1} > p_{k}(n+1)^n.
\end{equation}
\end{theorem}
\begin{proof}
The $1$-colored
partitions are considered in Theorem \ref{Andrew}. 
Therefore, let $k \geq 2$. Then, \cite{BKRT21}
implies that
$p_2(n)$ is log-concave for $n \geq 6$ and 
$p_k(n)$ is log-concave for $k\geq 3$ and all $n$. 
Since $x_{5
}^{\sigma}<2
$, we can apply Proposition
\ref{prop: strategy} and conclude
the proof.
\end{proof}
\begin{corollary}
Let $k \geq 2$. Then, the strict inequality (\ref{discrete:result})
holds true for all natural numbers $n$
except for $k=2$ and $n=1$ holds the opposite inequality
and for $k=3$ and $n=1$ holds equality.
\end{corollary}

\begin{proof}
For $k=2
$ we can check the cases $1\leq n
\leq 4
$
directly, also for $k=3$ the case $n=1$ and then using
that $x_{2}^{\sigma }<2$,
and for $k\geq 4$, we use $x_{1}^{\sigma }=3
$.
\end{proof}

\begin{remarks*}\
\begin{itemize}
\item[(i)] Let $k \geq 2$. Then, the sequence $
\left\{ \sqrt[n]{p_{k}\left( n\right) }\right\}
_{n\geq 1}$ is strictly decreasing
except in the cases $k\in \left\{ 2,3\right\} $ only
for $n\geq 2$. 
\item[(ii)] Note, it is not necessary to have strict log-concavity
\cite{BKPR25}. Moreover, the inequality 
$p_2(5)^2 < p_2(4) \, p_2(6)$ has no negative effect. 
\item[(iii)] The general result of \cite{BKRT21} on the log-concavity of $P_n^{\sigma}(x)$ is not sufficient to prove a result for $x \geq 2$ and $x \not\in \mathbb{N}$, since one also needs the initial condition related to
Proposition \ref{prop: strategy} following the proof method of \cite{WZ14}.
\item[(iv)] A direct proof of the claim $x_n^{\sigma}<2
$ for
$2\leq n \leq 5
$ would give a new proof of Theorem \ref{th:color} and the Conjecture.
\end{itemize}
\end{remarks*}

\section{The
cases $\psi_{\ell}(n)$: Pochhammer und Laguerre polynomials}
\label{sect:poch}
For $\ell=0$, the most explicit case, we have a closed formula for
the polynomials $P_n^{\psi_0}\left( x\right)
$ given by (\ref{Pochhammer}). 
These are the Pochhammer polynomials.
This makes it possible to determine $x_n^{\psi_0}$. Associated Laguerre polynomials are obtained for $\ell=1$. To determine the location of
$x_n^{\psi_1}$ is already a non-trivial task. It is an open question how to
attack the cases $\ell>1$.
\subsection{Largest real zero of $\Delta_n^{\psi_0}(x)$}
\label{ell=0}
\begin{proposition}
Let $g\left( n\right) =\psi _{0}\left( n\right) =1$ for all $n$. Then the
largest zero of
\begin{equation*}
\Delta_n^{\psi_0}(x)\, = \,
\left( P_{n}^{\psi _{0}}\left( x\right) \right) ^{n+1}-\left( P_{n+1}^{\psi _{0}}\left( x\right) \right) ^{n}
\end{equation*}
is $x_n^{\psi_{0}}=1$ for all $n\geq 1$.
\end{proposition}
\begin{proof}
We obtain
\begin{equation*}
\Delta_n^{\psi _{0}}\left( x\right) =
\left(
\prod _{k=0}^{n-1}\frac{x+k}{k+1}\right) ^{n+1}-\left(
\prod _{k=0}^{n}\frac{x+k}{k+1}\right) ^{n}=\left( \prod _{k=0}^{n-1}\frac{x+k}{k+1}\right) ^{n}\left( \prod _{k=0}^{n-1}\frac{x+k}{k+1}-\left(
\frac{x+n}{n+1}\right) ^{n}\right).
\end{equation*}
Therefore, its zeros are $0,-1,\ldots ,1-n$,
each $n$fold and the zeros of
$$\Pi
_{n}\left( x\right) =
\prod _{k=0}^{n-1}\left( \frac{x-1}{k+1}+1\right) -\left( \frac{x-1}{n+1}+1\right) ^{n}.$$
Note that it has a zero at $x=1$
since both terms
in the difference are $1$ in
this case. If $x>1$ and $k<n$,
then $\frac{x-1}{k+1}>\frac{x-1}{n+1}$.
Thus, $\Pi
_{n}\left( x\right) >0$,
confirming that $x=1$ is
the largest zero in this case.
\end{proof}
\subsection{Largest real zero of $\Delta_n^{\psi_1}(x)$}
\label{ell=1}
Let $g(n)=\psi_1(n)$. Then, the associated
polynomials $P_n^{\psi_1}(x)$
are
related to $1$-associated Laguerre polynomials $L_{n-1}^{(1)}(x)$. We refer to
Section \ref{sect:pola} and \cite{HNT22}.
\begin{proof}[Proof of Theorem \ref{th:ell=1}]
The polynomials $P_n^{\psi_1}(x)$ satisfy Tur\'{a}n inequalities (\cite{HNT22}, Theorem 1.4) for all non-negative real numbers $x$ and $n\geq 2$:
\begin{equation*}
P_n^{\psi_1}(x)^2 \geq P_{n-1}^{\psi_1}(x) \, P_{n+1}^{\psi_1}(x).
\end{equation*}
Nevertheless, the initial condition (\ref{special}) has some {\it delay\/}.
It is sufficient to consider the plot Figure \ref{figure:Laguerre}
with $1 \leq n \leq 25$. We have $x_1^{\psi_1}= \psi_1(2)=2$ and $x_n^{\psi_1}<2$ otherwise.
Therefore,
we can conclude the result for all $x \geq 2$ for $n
\geq 2$. Further, for all $n\geq 6$, we have $x_n^{\psi_1} <1$. 
\begin{figure}[H]
\includegraphics[width=.5\textwidth]{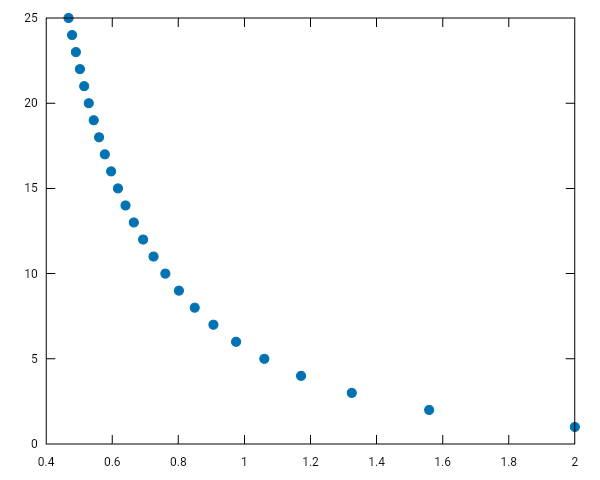}
\caption{\label{figure:Laguerre}Largest real zeros of 
$\Delta_n^{\psi_1}(x)$ for $1 \leq n \leq 25$}
\end{figure}
\end{proof}
\subsection{Largest real zero of $\Delta_n^{\psi_2}(x)$}
We add the case $g(n)=\psi_2(n)=n^2$, 
which shows the same pattern as for
$\psi _{1}\left( n\right) $.
\begin{figure}[H]
\includegraphics[width=.5\textwidth]{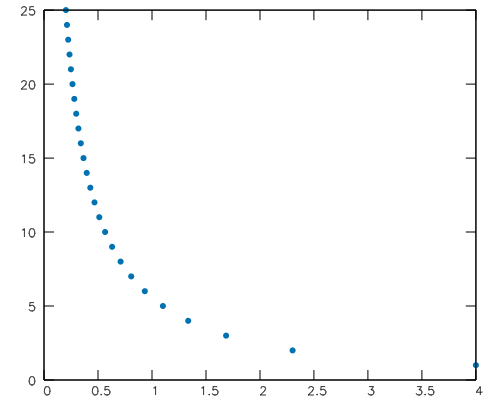}
\caption{\label{figure:psi2}Largest real zeros of 
$\Delta_n^{\psi_2}(x)$ for $1 \leq n \leq 25$}
\end{figure}

\section{Plane
partitions and
overpartitions}
\subsection{Proof of Theorem \ref{th:plane}}
\begin{proof}
We know already by Ono, Pujahari, and Rolen \cite{OPR22} 
that $\{pp(n)\}_n$ is log-concave for $n \geq 12$.
This has been previously conjectured by Tr{\"o}ger and the two authors of this paper (\cite{HNT22}, Section 1.2).
Since $x_n^{\sigma_2}$ is positive and smaller than one for $6 \leq n \leq 11
$
(see Figure \ref{figure:plane}) we can apply Proposition \ref{prop: strategy} and finalize the proof.
\end{proof}
It follows from Figure \ref{figure:plane} that 
\begin{equation} \label{plane:x}
P_n^{\sigma_2}(x)^{n+1} > P_{n+1}^{\sigma_2}(x)^{n} 
\end{equation}
for $x \geq 1$ and all $6
\leq n \leq 25$.
Therefore, (\cite{HNT22}, Conjecture 2 in Section 1.3.2) would imply
that the polynomization (\ref{plane:x})
of  Theorem \ref{th:plane} is also valid
for all $n \geq 6$ and $x \geq 1$. 
\subsection{Proof of Theorem \ref{th:over}}
\begin{proof}
The overpartition function $\bar{p}(n)$ 
is log-concave for all $n \in \mathbb{N}$.
Note that $\bar{p}(0)=1$,  $\bar{p}(1)=2$ and $\bar{p}(2)=4$.
\begin{theorem}[Engel \cite{En17}]
The function $\bar{p}(n)$ is log-concave for $n \geq 1$.
\end{theorem}
We have that $x_1^{\bar{g}}=x_2^{\bar{g}}=2$ and $x_n^{\bar{g}}<2$ 
for $ 3 \leq n \leq 25$.
Again, applying Proposition \ref{special} proves
Theorem~\ref{th:over}.
\end{proof}

\section{Commuting $\ell$-tuples of
permutations}
We consider the double sequence $\{N_{\ell}(n)\}$ of positive integers
provided by properties of the symmetric group $S_n$ acting
on $n$ elements and its abelian subgroups,
as measured
by its generators. We refer to
\cite{Ab24,ABDV24}.
Let $\ell,n  \geq 1$ and 
\begin{equation*}
C_{\ell,n}:= \left\{ \pi=(\pi_1, \ldots, \pi_{\ell}) \in S_n^{\ell}\, : \,
\pi_j \pi_k = \pi_k \pi_j \text{ for } 1 \leq j,k \leq \ell \right\}
.
\end{equation*}
Each
$\pi \in C_{\ell,n}$ corresponds to a group homomorphism from
$\mathbb{Z}^{\ell} $ to $S_n$. We
then define
\begin{equation*}
N_{\ell}(n):= \frac{\vert C_{\ell,n} \vert}{n!} 
=
\frac{\left\vert 
\func{Hom}\left( \mathbb{Z}^{\ell},S_n\right) \right\vert
}{n!}  \in \mathbb{N}, \text{ with }  N_{\ell}(0):=1.
\end{equation*}
Dey \cite{De65} and Wohlfahrt \cite{Wo77}
offer a compelling
recursive approach
determining
the number of homomorphisms $n! \, N_G(n):=  \left\vert \func{Hom}\left( G,S_n\right) \right\vert$ 
of a finitely generated group $G$ into $S_n$. Let $g_G(n)$ denote the
number of subgroups of $G$
with index $n$.
Notably, there has been
extensive research
by Lubotzky and Segal \cite{LS03}
studying the growth of these numbers in
the context of
characterizing the group $G$ itself:
\begin{equation} 
\label{eq:Wohlfahrt}
N_G(n) = \frac{1}{n} \sum_{k=1}^n g_G(k) \, N_G(n-k), 
\end{equation}
with the
initial value $N_G(0):=1$. 
Let $G=\mathbb{Z}^{\ell}$
and $g_{\ell}(n):= g_G(n)$. Then $g_1(n)=1$ and $g_2(n) = \sigma(n) = \sum_{d
\mid n} d$.
Therefore, we recover in the case $\ell=2$ from (\ref{eq:Wohlfahrt})  Euler's
well known recursion formula for the partition numbers $p(n)$ \cite{An98, On04}.
The case $\ell =3$ reveals a connection to topology,
as highlighted by Britnell (see the introduction of \cite{Br13}) due to work by
Liskovets and Medynkh \cite{LM09}. They discovered that  
$N_{3}(n)$ counts the number
of non-equivalent $n$-sheeted coverings of a torus.
Additionally, 
the numbers $N_{\ell}(n)$ were
identified by Bryan and Fulman \cite{BF98}
as the $n$th orbifold characteristics. This generalizes
the work by
Hirzebruch and H\"ofer \cite{HH90} concerning the ordinary
and string-theoretic Euler characteristics of symmetric products.

To illustrate the growth of $\{N_{\ell}(n)\}_{\ell,n}$,
we present
the first values for $0 \leq n \leq 10$ and  $2 \leq \ell \leq 4$ in Table \ref{tab:Nell}. 
Note that $N_{1}(n)=1$.
%
\begin{table}[H]
\[
\begin{array}{cccccccccccc}
\hline
n & 0 & 1 & 2 & 3 & 4 & 5 & 6 & 7 & 8 & 9 & 10 \\ \hline \hline
N_2(n) & 1 & 1 & 2 &
3 &
5 &
7 &
11 &
15 &
22 & 30 & 42 \\
N_3(n) & 1 & 1 &4&8&21&39&92&170&360&667&1316 \\
N_4(n) &1 & 1 &8&21&84&206&717&1810&5462&13859&38497 \\
\hline
\end{array}
\]
\caption{\label{tab:Nell}
Values for $ 0 \leq n \leq 10$
}
\end{table}

Let $g_{\ell}(n)$ be the number of subgroups of $\mathbb{Z}^{\ell}$ of index $n$. In this subsection,
we examine
the cases $\ell=3$
and $\ell=4$. See Figure \ref{figure:malek3} and Figure \ref{fig:malek4}.
\begin{figure}[H]
\includegraphics[width=.5\textwidth]{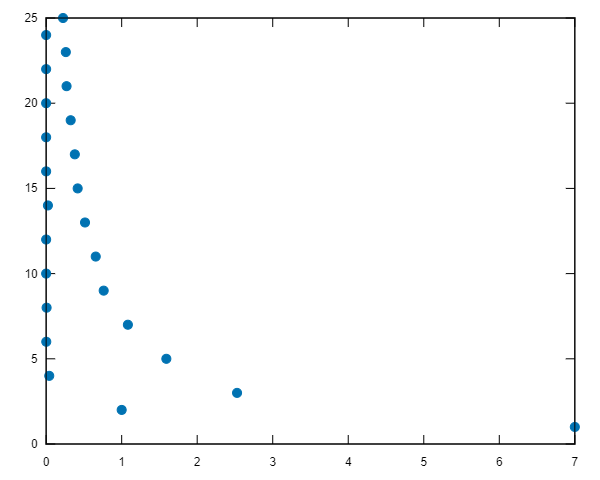}
\caption{\label{figure:malek3}Largest real zeros of 
$\Delta _{n}^{g_{3}
}\left( x\right) =
\left( P_{n}^{g_{3}
}\left( x\right) \right) ^{n+1}-\left( P_{n+1}^{g_
{3}}\left( x\right) \right) ^{n}$ with $n$ vertical}
\end{figure}

\begin{figure}[H]
\includegraphics[width=.5\textwidth]{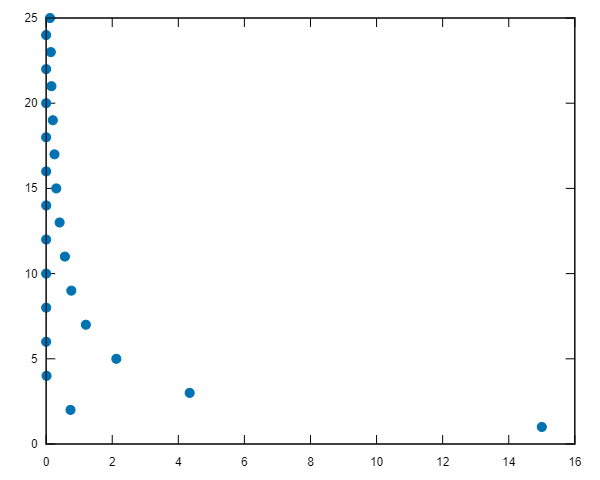}
\caption{\label{fig:malek4}Largest real zeros of $\Delta _{n}^{g_{4}}\left( x\right) =\left( P_{n}^{g_{4}}\left( x\right) \right) ^{n+1}-\left( P_{n+1}^{g_{4}}\left( x\right) \right) ^{n}$ with $n$ vertical}
\end{figure}

\begin{theorem} 
Let $8
\leq n \leq 25$. Then $\Delta _{n}^{g_3}\left( x\right)>
0$ for 
all $x \geq 1$. Let $8
\leq n \leq 25$.
Then
$\Delta _{n}^{g_4}\left( x\right)>
0$
for all $x\geq 1$.
\end{theorem}

Ultimately,
we state the following speculation
which
appears to be evident
from Figure \ref{figure:malek3} and Figure \ref{fig:malek4}.

\begin{conjecture} Let $a_3=
8$ and $a_4=8
$. Let $\ell=3$
or $\ell=4$.
Then $\Delta _{n}^{g_{\ell}}\left( x\right)
>0
$ for all $x\geq 1$ for $n\geq
a_{\ell }$.
\end{conjecture}

%
%


{\bf Acknowledgments.}
The authors thank Shane Chern for informing us
about the work by
Wang and Zhu
and Chen and Zheng, after a first version of the paper had been distributed
and Steven Charlton for supporting us with the computations of
values of the $x_{n}^{\sigma }$.

\end{document}